\documentclass[12pt]{amsart}
\usepackage{amssymb,mathrsfs,color,bbold}
 \usepackage[all]{xy}
\usepackage{color}

\theoremstyle{plain}
\newtheorem{theorem}{Theorem}[section]
\newtheorem{proposition}[theorem]{Proposition}
\newtheorem{corollary}[theorem]{Corollary}
\newtheorem{lemma}[theorem]{Lemma}

\theoremstyle{definition}
\newtheorem{remark}[theorem]{Remark}
\newtheorem{example}[theorem]{Example}

\newtheorem{problem}[theorem]{Problem}

\newcommand{\abs}[1]{\lvert#1\rvert}
\newcommand{\norm}[1]{\lVert#1\rVert}

\newcommand{\N}{{\mathbb N}}

\newcommand{\ol}{\overline}

\DeclareMathOperator{\Span}{Span}

\DeclareMathOperator{\supp}{supp}
\def\one{\mathbb 1}

\author[N.~Gao]{Niushan Gao}
\address{Department of Mathematics and Computer Science, 
University of Lethbridge, Lethbridge, Canada T1K 3M4}
\email{gao.niushan@uleth.ca}

\author[D.~Leung]{Denny H.~Leung}
\address{Department of Mathematics, National University of Singapore, Singapore
117543}
\email{matlhh@nus.edu.sg}


\thanks{The first author is a PIMS Postdoctoral Fellow. He also acknowledges 
support from  the National Natural Science Foundation of China (No.~11601443).
The second author is partially supported by AcRF grant R-146-000-242-114.}

\date{\today}

\begin{document}

\title[Smallest order closed sublattices and option spanning]{Smallest order
closed sublattices and option spanning}
 \maketitle

\begin{abstract}
Let $Y$ be a sublattice of a vector lattice $X$.
We consider the problem of identifying the smallest order closed sublattice of
$X$ containing $Y$.  
It is known that the analogy with topological closure fails. Let $\overline{Y}^o$ be
the order closure of $Y$ consisting of all order limits of nets of
elements from $Y$.  
Then $\overline{Y}^o$ need not be order closed. 
We show that in many cases the smallest order closed sublattice containing $Y$
is in fact the 
second order closure $\overline{\overline{Y}^o}^o$.
Moreover, if $X$ is a $\sigma$-order complete Banach lattice, then the condition
that $\overline{Y}^o$ is order closed for every sublattice $Y$ characterizes order
continuity of the norm of $X$.
The present paper provides a general approach to  a fundamental result in
financial economics 
concerning the spanning power of options written on a financial asset.
\end{abstract}

\section{Introduction}
\subsection{Motivations}
Let $\Omega$ be a finite set standing for the state space of a static financial
market at the terminal date. 
A financial asset in the market is represented by a function $f$ on $\Omega$.
The call
(respectively, put) option written on an asset $f$ with strike price
$k\in\mathbb{R}$ can be represented as 
$(f-k\one)^+$ (respectively, $(k\one-f)^+$). Here $\one$ denotes the constant
one function on $\Omega$.
In the seminal paper \cite{R:76}, Ross showed that, if the underlying asset
separates
states of the market at the terminal date, then options on this asset generate
complete
markets, i.e., every contingent claim is replicated by a portfolio of some call
and put options on this asset. Mathematically speaking, it means that, for any
injective
function $f\in \mathbb{R}^\Omega$,
$$\mathbb{R}^\Omega=\Span\big\{(f-k\one)^+,(k\one-f)^+:k\in\mathbb{R}\big\}.$$
This notion that options complete markets, pioneered by Ross, is at
the core of modern financial economics (\cite{BM:00}), and has been under
extensive exploration.

In particular, Ross's result has been extended to financial markets with
infinite state spaces. Let $(\Omega,\Sigma,\mathbb{P})$ be a probability space.
For an asset $f \in L^p(\Sigma)$, $1\leq p\leq \infty$, its \emph{option space}
is defined by 
$$O_f:=\Span\big\{(f-k\one)^+,(k\one-f)^+:k\in\mathbb{R}\big\}.$$
Through the work of Nachman \cite{N:88}, Galvani \cite{G:09}, and Galvani and
Troitsky \cite{GT:10}, it is established that, if $f$ is of \emph{limited
liability}, i.e., $f\geq 0$ a.s., then
$$\overline{O_f}^{a.s.}\cap
L^p(\Sigma)=\overline{O_f}^{\norm{\cdot}_p}=L^p(\sigma(f)),\mbox{ for }1\leq
p<\infty,$$
$$\overline{O_f}^{a.s.}\cap
L^\infty(\Sigma)=\overline{O_f}^{\sigma(L^\infty,L^1)}=L^\infty(\sigma(f)).$$
Here $\overline{O_f}^{a.s.}$ is the collection of all random variables that are
almost sure limits of sequences in $O_f$, and $\sigma(f)$ is the
$\sigma$-algebra generated by $f$. 
Recently, these results have been generalized in \cite{GX:17} to model spaces
beyond $L^p$, using the topology $\sigma(X,X_n^\sim)$, where
$X_n^\sim$ is order continuous
dual of $X$. 
Specifically, let $X$ be an ideal (i.e., solid subspace) of $L^0(\Sigma)$ that
contains the constant one function and admits a strictly positive order
continuous functional. Then for any limited liability asset $f\in X_+$, it holds
that
$$\overline{O_f}^{a.s.}\cap
X=\overline{O_f}^{\sigma(X,X_n^\sim)}=X(\sigma(f)).\eqno(*)$$
Here $X(\sigma(f))$ is the set of all random variables in $X$ that are
$\sigma(f)$-measurable, and is interpreted as
the collection of all financial claims written on the asset $f$, among which
options are obviously the  basic ones. 
Thus $(*)$ asserts that every claim written on the asset $f$ is the a.s.-limit of a sequence of portfolios of options on $f$.
It deserves mentioning that these
spanning properties of options played
a very useful role in the study of price extensions in
\cite{BR:91,GX:17,N:88}.

A fundamental fact used to prove $(*)$ is a beautiful theorem due to the
economists Brown and
Ross (\cite[Theorem (1)]{BR:91}), which asserts that, for any $0\leq s\leq b$ in
a uniformly complete vector lattice $X$,
$\Span\{(s-kb)^+,(kb-s)^+:k\in\mathbb{R}\}$ is the smallest sublattice of $X$
containing $s,b$. This implies in particular that the option space $O_f$  of any
limited liability asset $f$ is a sublattice (see Lemma~\ref{subl} below).
A close look at the proof of $(*)$ reveals that the terms in $(*)$
are precisely the smallest order closed sublattice of $X$ containing $O_f$. This
motivates us to investigate the smallest order closed sublattice containing a
given sublattice. Our study provides a general approach to the spanning power of
options.

The paper is structured as follows. In Section~\ref{main-result}, we prove that,
in many Banach lattices, the smallest order closed
sublattice containing a given sublattice $Y$ coincides with the uo-closure of
$Y$ as well as the second order closure of $Y$ (Theorem~\ref{uo-closure}).  It
is also shown
that if (and only if) the ({\em first}) order closure of any sublattice $Y$ in a
$\sigma$-order complete Banach lattice $X$ is order closed, then $X$ is order
continuous (Theorem~\ref{o-closure}).
On the other hand, Theorem~\ref{option} shows  that for a large class of Banach
function
spaces, the 
order closure of the option space $O_f$ is already order closed for all $f\geq 0$.  In a similar
vein, Theorem~\ref{o-closure-reg} shows that the order closure of any regular
sublattice of a
vector lattice is order closed.  These results show that the behavior of the
order closure of a sublattice can be quite subtle.
In Section~\ref{measurability}, we relate order closure to measurability,
following the approach
of Luxemburg and de Pagter \cite{DL:05,DL:05b}. 
Corollary~\ref{option-strong}
shows that options on a limited liability asset often have the strong
spanning power
that every  claim written on the asset is the \emph{order} limit of a
sequence of portfolios of options.

\subsection{Notation and Facts}
We adopt \cite{AB:06,AB:78} as standard references on unexplained terminology
and
facts
on vector and Banach lattices. For general facts about uo-convergence we refer
the reader to \cite{GTX:16} and the references therein. 
A net $(x_\alpha)_{\alpha\in\Gamma}$ in a vector
lattice $X$ is said to
\emph{order converge} to $x\in X$, written as $x_\alpha\xrightarrow{o}x$,
if there exists another net $(a_\gamma)_{\gamma\in \Lambda}$ in $X$ satisfying
$a_\gamma\downarrow 0$ and for any $\gamma\in\Lambda$ there exists $\alpha_0\in
\Gamma$ such that $\abs{x_\alpha-x}\leq a_\gamma$ for all $\alpha\geq
\alpha_0$; 
$(x_\alpha)$ is said to \emph{unbounded
order converge} (uo-converge for short) to $x\in X$, written as
$x_\alpha\xrightarrow{uo}x$, if $\abs{x_\alpha-x}\wedge y\xrightarrow{o}0$ for
any $y\in X_+$.
It is well known that, for a sequence $(f_n)$ in a  function space $X$,
$f_n\xrightarrow{o}0$ in $X$
iff $f_n\xrightarrow{a.s.}0$ and there exists $F\in X$ such that $\abs{f_n}\leq
F$ for all $n\geq 1$, and
$f_n\xrightarrow{uo}0$ in $X$ iff $f_n\xrightarrow{a.s.}0$.
Recall that a Banach lattice is \emph{order continuous} if
$\norm{x_\alpha}\rightarrow0$ whenever $x_\alpha\xrightarrow{o}0$. The
\emph{order continuous
dual} $X_n^\sim$ of a vector lattice $X$ is the collection of all linear
functionals $\phi$ which are order continuous, i.e., $\phi(x_\alpha)\rightarrow
0$
whenever
$x_\alpha\xrightarrow{o} 0$ in $X$. If $X$ is a Banach lattice, $X_n^\sim$ is a
band in $X^*$.
A \emph{Banach function space} over a probability space
$(\Omega,\Sigma,\mathbb{P})$ is an ideal of $L^0(\Sigma)$
with a complete norm such that $\norm{f}\leq \norm{g}$ whenever
$\abs{f}\leq\abs{g}$.
Every Banach function space has a separating order continuous dual
(\cite[Theorem~5.25]{AB:02}) and has \emph{the countable sup property}, i.e.,
every
set having a supremum admits a countable subset with the same supremum
(\cite[Lemma~2.6.1]{MN:91}).

Let $X$ be a vector lattice. For any $x,y\in X_+$, denote by $L_{x,y}$ the
smallest sublattice containing
$x,y$. Recall that Banach lattices and $\sigma$-order complete vector lattices
are
uniformly complete. Thus the following lemma applies to them.
\begin{lemma}\label{subl}
For any $x,y\geq 0$ in a uniformly complete vector lattice $X$, 
$$L_{x,y}=\Span\big\{(x-ky)^+,(ky-x)^+:k\in \mathbb{R}\big\}.$$
\end{lemma}
\begin{proof}
Note that both sides remain the same when we replace $y$ by $x+y$. Now apply
\cite[Theorem (1)]{BR:91} with $s=x$ and $b=x+y$.
\end{proof}

\section{Main Results}\label{main-result}

For a subset $A$ of a vector lattice $X$, we define its \emph{order closure}
(abbr.~o-closure)
$\overline{A}^{o}$ to be the collection of all $x\in X$ such that $x_\alpha
\xrightarrow{o}x$ in $X$ for some net $(x_\alpha)$ in $ A$.
We say that $A$ is \emph{order closed} (abbr.~o-closed)  in $X$ if
$A=\overline{A}^{o}$. 
We similarly define uo-closure and uo-closedness of a given subset.
Since lattice operations are both order continuous and uo-continuous, it is easy
to see that the o- and uo-closures of a sublattice remain sublattices.
However, the order closure of a sublattice need not be order closed.  This
is the main subject of investigation in this paper.

\begin{lemma}\label{closures}
Let $Y$ be a sublattice of a vector lattice $X$ and $I$ be an ideal of
$X_n^\sim$. Then $\overline{Y}^o\subset
\overline{Y}^{uo}\subset \overline{\overline{Y}^o}^o\subset
\overline{Y}^{\sigma(X,I)}$. Moreover,
\begin{enumerate}
 \item\label{closures1} if $\overline{Y}^o$ is order closed, then it is the
smallest order closed
sublattice of $X$ containing $Y$, and $\overline{Y}^o= \overline{Y}^{uo}$;
\item\label{closures2}  if
$\overline{Y}^{uo}$ is order closed, then it is the smallest order closed
sublattice of $X$ containing $Y$, and
$\overline{Y}^{uo}=\overline{\overline{Y}^o}^o$;
\item\label{closures3} if, in addition, $I$ separates points of $X$, then
$\overline{Y}^{\sigma(X,I)}$
is an order closed sublattice containing $Y$.
\end{enumerate}
\end{lemma}

\begin{proof}
Obviously, $\ol{Y}^o \subseteq \ol{Y}^{uo}$. Since $I\subseteq X_n^\sim$, $\ol{Y}^{\sigma(X,I)}$ is order closed in $X$. In particular,
$\ol{\ol{Y}^o}^o\subseteq \ol{Y}^{\sigma(X,I)}$.
Let $(y_\alpha)$ be
a net in $Y$ such that $y_\alpha\xrightarrow{ uo}x$ in $X$. By considering
the
positive and negative parts, respectively, we may assume that $(y_\alpha)\subset
Y_+$ and
$x\geq 0$. For each fixed $\beta$, it follows from $\abs{y_\alpha\wedge
y_\beta-x\wedge y_\beta 
}\leq \abs{y_\alpha-x}\wedge y_\beta$ that $y_\alpha\wedge
y_\beta\xrightarrow{o} y_\beta\wedge x$ in $X$, and
consequently, $y_\beta\wedge x\in
\overline{Y}^o$. By
$\abs{y_\beta\wedge x-x}\leq \abs{y_\beta-x}\wedge x$, it follows that
$y_\beta\wedge
x\xrightarrow{o}x$ in $X$, and therefore, $x\in
\overline{\overline{Y}^o}^o$. 
This proves that $\overline{Y}^{uo}\subset \overline{\overline{Y}^o}^o$. Items (1)
and (2) are now clear.
Suppose that $I$ separates points of $X$.
By \cite[Theorem~3.50]{AB:06}, the topological dual of $X$ under
$\abs{\sigma}(X,I)$ is precisely $I$, and thus by Mazur's
Theorem, $$
\overline{Y}^{\sigma(X,I)}=\overline{Y}^{\abs{\sigma}(X,I)}.\eqno(*)$$
This implies  that $\overline{Y}^{\sigma(X,I)}$ is a sublattice
of $X$ by \cite[Theorem~3.46]{AB:06}.
\end{proof}

Remark that \cite[Proposition~3.15]{GTX:16}, which asserts that a sublattice is
o-closed
iff
it is uo-closed, immediately follows from Lemma~\ref{closures}.

\begin{theorem}\label{uo-closure}
Let $X$ be a Banach lattice, $Y$ be a sublattice of $X$, and $I$ be an ideal of
$X_n^\sim$ separating points of $X$. 
Suppose that $X$ has the countable sup property.
Then $\overline{Y}^{uo}=\overline{\overline{Y}^{o}}^{\rm
o}=\overline{Y}^{\sigma(X,I)}$, and all of them are the
smallest order closed sublattice in $X$ containing $Y$.
\end{theorem}

\begin{proof}
In view of Lemma~\ref{closures} and $(*)$, it suffices to show that 
$\overline{Y}^{\abs{\sigma}(X,I)}\subset \overline{Y}^{uo}$.
Recall that the order completion, $X^\delta$, of $X$ is also a Banach lattice
having
the countable sup property. 
Note also that each member in $I$ extends uniquely to an order continuous
functional
on $X^\delta$ (\cite[Theorem~1.65]{AB:06}) and that
the collection of those extended functionals is an ideal of
$(X^\delta)_n^\sim$ separating points of $X^\delta$. 
Moreover, a net in $X$ is uo-null in $X$ iff it is uo-null in $X^\delta$
(cf.~\cite[Theorem~3.2]{GTX:16}). 
Thus, by passing to $X^\delta$, one may assume that $X$ is order complete.

Recall that if $0 \leq \phi\in X_n^\sim$, its null ideal and carrier
are defined, respectively, by
\[ N_\phi = \{x\in X: \phi(|x|) = 0\} \quad\text{and} \quad C_\phi = N_\phi^d.\]

\underline{Claim 1}.
Every sequence $(x_n)$ in $X_+$ is contained in $C_\phi$ for some $\phi\in
\ol{I}_+$.

Indeed, for each $\phi\in I_+$, let $P_\phi$ be the band projection onto
$C_\phi$.
For each $n $, $(P_\phi x_n)_\phi$ is an upwards directed net, bounded above by
$x_n$.
Since, for any $\psi\in I_+$, $\psi(x_n-\sup_{\phi\in I_+}P_\phi x_n) \leq
\psi(x_n-P_\psi x_n) = 0$, and $I$ separates points of $X$,  it follows that $x_n=\sup_{\phi\in I_+}P_\phi x_n$.
As $X$ has the countable sup property, there exists a sequence $(\phi^n_m)_m$ in
$I_+$ such that
$x_n = \sup_{m}P_{\phi^n_m} x_n$.  Let $\phi =
\sum_{m,n}\frac{\phi^n_m}{2^{m+n}\|\phi^n_m\|+1}$.
Then $0\leq \phi\in \ol{I}$. 
Since $P_{\phi^n_m}x_n \in C_{\phi^n_m} \subseteq C_\phi$ for all $m,n$, and
$C_\phi$ is a band, we see that $x_n \in C_\phi$ for all $n$.  Thus the claim is
proved.

\underline{Claim 2}. If $(x_n)$ is an order bounded sequence in $C_\phi$  for
some $0\leq \phi \in X_n^\sim$  and
$\sum\phi(|x_n|)<\infty$, then $(x_n)$ order converges to $0$.

Set $u =  \inf_k\sup_{n\geq k}|x_n|$.  Since $\phi$ is order continuous,
$\phi(u) \leq \sum_{n\geq k}\phi(|x_n|)$ for all $k$.
Hence, $\phi(u) = 0$.
Also, $u\in C_\phi$, since $C_\phi$ is a band.  It follows that $u =0$.
Therefore, $(x_n)$ order converges to $0$, and the claim is proved.

Suppose that $0\leq x\in \ol{Y}^{|\sigma|(X,I)}$. By Claim 1, choose $ \phi\in
\ol{I}_+$ such that $x\in C_\phi$.
Given any $n\in \N$, choose $\psi\in I_+$ such that $\|\phi-\psi\|<
\frac{1}{2^n\|x\| +1}$, and choose $y_n\in Y_+$ such that $\psi(|y_n-x|)<
\frac{1}{2^n}$.
Then
\[ \phi(|y_n-x|\wedge x) \leq \|\phi-\psi\| \|x\| + \psi(|y_n-x|) \leq
\frac{2}{2^n}.\]
It follows by Claim 2 that $(|y_n-x|\wedge x)$ order converges to $0$.
Now choose $\phi'\in \ol{I}_+$ such that $x, y_n \in C_{\phi'}$ for all $n$.
Since $(|y_n-x|\wedge x)$ order converges to $0$ and $\phi'$ is order
continuous, we may assume that $\phi'(|y_n-x|\wedge x) \leq\frac{1}{2^n}$ for
all $n$.
As above, for each $n$, there exists $z_n \in Y_+$ so that 
$\phi'(|z_n-x| \wedge y_n) \leq \frac{2}{2^n}$.
For any $w\in X_+$,
\begin{align*}
 \phi'(|z_n\wedge y_n-x| \wedge w) &\leq \phi'(|z_n\wedge y_n - x\wedge y_n|)  +
\phi'(|x\wedge y_n -x |)\\
 & \leq  \phi'(|z_n-x|\wedge y_n)  + \phi'(x\wedge |y_n -x|)\leq \frac{3}{2^n}
 \end{align*}
for all $n$.  By Claim 2, $(|z_n\wedge y_n-x| \wedge w)$ order converges to $0$.
This proves that $(z_n\wedge y_n)$ uo-converges to $x$.  Therefore, $x\in
\ol{Y}^{uo}$.
\end{proof}

Clearly, Theorem~\ref{uo-closure} applies to Banach function spaces over
probability spaces.

\begin{remark}\label{uo-cl-re}
\begin{enumerate}
\item\label{uo-cl-re1} Our proof yields that under the assumptions of
Theorem~\ref{uo-closure} if $x\in\overline{Y}^{uo}$ then there exists a sequence
in $Y$ uo-converging to $x$.
\item
The conclusion of Theorem~\ref{uo-closure} still holds if $X$ is merely a vector
lattice but $I$ contains a
strictly positive order continuous functional $\phi$ on $X$. 
\end{enumerate}
\end{remark}

\begin{remark}
\begin{enumerate}
\item Theorem~\ref{uo-closure} implies in particular that
$\overline{Y}^{\sigma(X,I)}$ may be independent of $I$ when $Y$  is a
\emph{sublattice}. This suggests that topological properties may improve
significantly when order structures are involved.
\item
View $\ell^\infty$ as the dual space of $\ell^1$.
For a subset $A$ in $\ell^\infty$, denote by $\overline{A}^{(1)}$ its
$w^*$-sequential closure, and by $\overline{A}^{(n+1)}$ the
$w^*$-sequential
closure of $\overline{A}^{(n)}$ for $n\geq 1$. 
Note that $\overline{A}^o=\overline{A}^{(1)}$ for any subset $A$ in
$\ell^\infty$.
Indeed, if $a_n\xrightarrow{w^*} x$, then $(a_n)$ is
bounded in
$\ell^\infty$ and converges to $x$ coordinatewise, so that  $a_n\xrightarrow{o}
x$ in $\ell^\infty$. Conversely, if a net in $A$ order converges $x$, then by
passing to a tail, we may assume that it is bounded. Clearly, we can extract a
sequence out of it. which converges  to
$x$ coordinatewise, and thus, in $w^*$ by Lebesgue Dominated Convergence
Theorem.
This observation, together with Theorem~\ref{o-closure2} (applied with
$I=(\ell^\infty)_n^\sim=\ell^1$), implies that
$$\overline{Y}^{(2)}=\overline{Y}^{w^*}$$
for any sublattice $Y$ of $\ell^\infty$. This is in sharp contrast to
Ostrovskii's Theorem (cf.~\cite[Theorem~2.34]{HZ:07}), which implies that
$\ell^\infty$ has a subspace $W$ such that $$\overline{W}^{(1)}\subsetneqq
\overline{W}^{(2)}
\subsetneqq\cdots\subsetneqq\overline{W}^{w^*}.$$
Again, it suggests that order structures improve
topological properties.
\end{enumerate}
\end{remark}

\begin{problem}\label{problem}
Is $\overline{Y}^{uo}$ order closed for every sublattice of a vector lattice
$X$?
\end{problem}

If we consider $\ol{Y}^o$ instead of $\ol{Y}^{uo}$ in Problem~\ref{problem},
then it
turns out that an affirmative answer to the problem characterizes order
continuity of $X$.
We begin with a lemma.

\begin{lemma}\label{o-not-o}
There exist $u,v>0$ in $\ell^\infty$ such that $\overline{L_{u,v}}^o\neq
\overline{L_{u,v}}^{uo}$.
\end{lemma}

\begin{proof}
Regard $\ell^\infty$ as $ \ell^\infty(\mathbb{N}\times \mathbb{N})$, and write
every element $x\in \ell^\infty(\mathbb{N}\times \mathbb{N})$ as $x =
(x_{mn})_{m,n\geq 1}$, where $x_{mn} \in \mathbb{R}$ for all $m,n \geq 1$.
Choose strictly increasing sequences $(c_{mn})^\infty_{n=1}$, $m\in \N$, 
and $(c_m)^\infty_{m=1}$ such  that
\begin{enumerate}
\item $(c_{mn})^\infty_{n=1}$ converges to $c_m$ for all $m$,
 \item $0 < c_m <c_{m+1, n}<1$ for all $m,n$.
\end{enumerate}
Let  $u = (u_{mn})\in\ell^\infty$  and $v =
(v_{mn})\in\ell^\infty$, where $u_{m1} = \frac{1}{m}$ and $u_{m,n+1} = 1$,
$v_{m1} =
\frac{c_m}{m}$ and $v_{m,n+1} = c_{mn}$ for all $m,n\geq 1$.

For any $k,j\in \N$, if $c_{kj} < \alpha < \alpha' < c_{k,j+1}$ and $c_k < \beta
< \beta' < c_{k+1,1}$,
then a direct calculation shows that if we write $x^{kj} = (x^{kj}_{mn})$ for
the
element
 \[\frac{(v-\beta u)^+ - (v-\beta'
u)^+}{\beta-\beta'}- \frac{(v-\alpha u)^+ - (v-\alpha'u)^+}{\alpha-\alpha'},\]
then 
\[ kx^{kj}_{k1} =1,\ kx^{kj}_{kn} = 0 \text{ if $2\leq n \leq j+1$},\ \text{and
$x^{kj}_{mn} = 0$ if $m \neq k$}.
\]
Let $y^j = \sum^j_{k=1}kx^{kj}$.  
Then $y^j \in L_{u,v}$, and $(y^j)$ converges coordinatewise to the element
$e\in
\ell^\infty(\N\times \N)$ given by $e_{mn} =1$ if $n=1$ and $0$ otherwise.
Thus $e\in \ol{L_{u,v}}^{uo}$.

We now show that $e\not\in \overline{L_{u,v}}^{o}$.  Otherwise, we can find an
order,
hence norm, bounded sequence  $(z^{(N)})$ in $L_{u,v}$ such that
$\lim_{N}z^{(N)}_{mn}=e_{mn}$ for any $m,n\geq 1$. 
For any $m\geq 2$, we can choose $N$ large enough such that 
$$\abs{z^{(N)}_{m1}-1}<\frac{1}{2},\mbox{ so that }z^{(N)}_{m1}>\frac{1}{2}.$$
Observe that $\lim_nz^{(N)}_{mn}=mz^{(N)}_{m1}$ since this holds for $u$ and $v$
and thus for every vector in $L_{u,v}$.
Thus, $\norm{z^{(N)}}_\infty\geq \frac{m}{2}$. By arbitrariness of $m$, this 
contradicts the boundedness of
$(z^{(N)})$. Therefore, $e\not\in \overline{Y}^{o}$, so that
$\overline{L_{u,v}}^o\neq
\overline{L_{u,v}}^{uo}$.
\end{proof}

Recall that a sublattice $Y$ of a vector lattice $X$ is said to be
\emph{regular}
if any net in $Y$ that decreases to $0$ in $Y$ also decreases to $0$ in $X$.

\begin{theorem}\label{o-closure}
Let $X$ be a $\sigma$-order complete Banach lattice.
The following are equivalent.
\begin{enumerate}
\item\label{o-closure1} $X$ is order continuous.
 \item\label{o-closure2} $\overline{Y}^o=\overline{Y}^{\sigma(X,X_n^\sim)}$ for
every sublattice
$Y$ of $X$.
\item\label{o-closure3} $\overline{Y}^o$ is order closed for every sublattice
$Y$ of $X$.
\item\label{o-closure4} $\overline{Y}^o=\overline{Y}^{uo}$ for every sublattice
$Y$ of $X$.
\item\label{o-closure5}
$\overline{L_{x,y}}^o=\overline{L_{x,y}}^{\sigma(X,X_n^\sim)}$ for
all $x,y\in X_+$.
\item\label{o-closure6} $\overline{L_{x,y}}^o$ is order closed for all $x,y\in
X_+$.
\item\label{o-closure7} $\overline{L_{x,y}}^o=\overline{L_{x,y}}^{uo}$ for all
$x,y\in X_+$.
\end{enumerate}
\end{theorem}

\begin{proof}
Suppose that \eqref{o-closure1} holds. 
Then every order convergent net is norm convergent. Note also that every norm
convergent sequence admits a subsequence order converging to the same limit
(cf.~\cite[Lemma~3.11]{GX:14}). Therefore, the order closure of any set
coincides with its norm closure.
Moreover, $\sigma(X,X_n^\sim)$ is now just the weak topology, and thus the
$\sigma(X,X_n^\sim)$-closure coincides with the weak closure. Hence,
\eqref{o-closure2} holds by Mazur's Theorem. The implication
\eqref{o-closure2}$\Rightarrow$\eqref{o-closure3} is immediate because the
$\sigma(X,X_n^\sim)$-closure of any set is order closed.
The implication \eqref{o-closure3}$\Rightarrow$\eqref{o-closure4} follows from
Lemma~\ref{closures}.
Similarly, we obtain
\eqref{o-closure1}$\Rightarrow$\eqref{o-closure5}$\Rightarrow$\eqref{o-closure6}
$\Rightarrow$\eqref{o-closure7}. Obviously, \eqref{o-closure4} implies
\eqref{o-closure7}.

It remains to be shown that \eqref{o-closure7}$\Rightarrow$\eqref{o-closure1}.
Suppose that
$X$ is not order continuous. Then $X$ has a lattice isomorphic copy of
$\ell^\infty$. 
The proof of \cite[Theorem~4.51]{AB:06} shows that the copy of
$\ell^\infty$ can be chosen to be regular in $X$. 
For a subset $W$ of $\ell^\infty \subseteq X$, denote its order closures in
$\ell^\infty$ and in $X$  by $\ol{W}^{o1}$ and $\ol{W}^{o2}$, respectively. 
Similarly for the respective uo-closures.
By Lemma 2.6, there are $u, v>0$ in $\ell^\infty$ and an element $e\in
\ell^\infty$ such that $e\in \ol{Y}^{uo1} \backslash \ol{Y}^{o1}$, where $Y =
L_{u,v}$.
We claim that $e\in \overline{Y}^{uo2}\backslash \overline{Y}^{o2}$.
Since $\ell^\infty$ is regular, every uo-null net in $\ell^\infty$ is uo-null in
$X$ by \cite[Theorem~3.2]{GTX:16}, implying that $e\in \overline{Y}^{uo1}\subset
\overline{Y}^{uo2}$.
If $e\in \overline{Y}^{o2}$, then there exists a net $(y_\alpha)$ in $Y$ such
that
$y_\alpha\xrightarrow{o}e$ in $X$.
By passing to a tail, we may assume that $(y_\alpha)$ is order, and thus norm,
bounded in $X$.
Then it is norm, and thus order, bounded in $\ell^\infty$.
By \cite[Corollary~2.12]{GTX:16}, we obtain that $y_\alpha\xrightarrow{o}e$ in
$\ell^\infty$, contradicting our choice of $e\not\in \overline{Y}^{o1}$.
This proves  \eqref{o-closure7}$\Rightarrow$\eqref{o-closure1}.
\end{proof}

The next main result (Theorem~\ref{option}) is a ``localized'' version of
Theorem~\ref{o-closure}.  It also yields information on the order closures of
option spaces in many
instances.
Recall first that the \emph{order continuous part}, $X^a$, of a Banach lattice
$X$ is the collection of all vectors $x$ in $X$ such that every disjoint sequence in
$[0,\abs{x}]$ is norm null. It is the largest norm closed ideal of $X$ which is
order continuous in its own right.
For a Banach function space $X$ defined on a probability space
$(\Omega,\Sigma,\mathbb{P})$, it is well known and easily seen that $\one\in
X^a$ iff  $X$ contains the constant
functions, and  
\[\lim_{\mathbb{P}(A)\to 0}\|\one_A\| = 0.\eqno(\diamond)\]

\begin{lemma}\label{split}
Let $X$ be a Banach function space over $(\Omega,\Sigma,\mathbb{P})$ such that
$\one\in X^a$, and $f\in X_+$. 
Let $g\in X_+$ be a bounded function that is the a.s.-limit of a sequence in
$O_f$.  For any
$\varepsilon > 0$, there exist  $h^1,h^2\in X_+$ and a set
$A \in \Sigma$ such that  $\mathbb{P}(\Omega\backslash A) < \varepsilon$, $\supp
h^1\subseteq  A$, $\supp
h^2\subseteq \Omega\backslash A$, $\|(g-h^1)\one_A\|_\infty <\varepsilon$,
$\|h^2\| < \varepsilon$ and $h^1 +
h^2 \in O_f$.
\end{lemma}

\begin{proof}
Assume that $0\leq g\leq \one$. 
Let $\varepsilon > 0$ be given.
By $(\diamond)$, there exists $\delta\in (0,\varepsilon)$ such that $\|\one_A\|<
\varepsilon$ whenever $A\in\Sigma$ and $\mathbb{P}(A)< \delta$.
Since $g$ is the a.s.-limit of a sequence  in $O_f$, 
by Egoroff's Theorem, there exist $h \in O_f$ and $A\in \Sigma$ such that
$$\|(g-h)\one_A\|_\infty < \varepsilon,\mbox{ and }\mathbb{P}(\Omega\backslash
A)<\delta.$$
Since $O_f$ is a sublattice containing $\one$, by replacing $h$ with $h^+\wedge
\one$, we may assume that $0\leq h\leq \one$.
Set $h^1 = h\one_A$ and $h^2 = h\one_{\Omega\backslash A}$.
Obviously, we have $\mathbb{P}(\Omega\backslash
A)<\varepsilon$,  $\supp
h^1\subseteq  A$, $\supp
h^2\subseteq \Omega\backslash A$, $\|(g-h^1)\one_A\|_\infty
<\varepsilon$ and $h^1 + h^2 \in O_f$.
Also, $\abs{h^2} \leq \one_{\Omega\backslash A}$ and hence
$\norm{h^2} \leq \norm{\one_{\Omega\backslash A}} < \varepsilon$ since 
$\mathbb{P}(\Omega\backslash
A)<\delta$.
\end{proof}

\begin{theorem}\label{option}
Let $X$ be a $\sigma$-order complete Banach lattice, and let $0<x\in X^a$. Then
$\overline{L_{x,y}}^o$ is order
closed for every $y\geq 0$. In particular, if $X$ is a Banach function space
over $(\Omega,\Sigma,\mathbb{P})$ such that $\one\in X^a$, then
$\overline{O_f}^o$ is order
closed for every $f\geq 0$.
\end{theorem}

\begin{proof}
We first prove the special case. Suppose $\one\in X^a$. In view of
Theorem~\ref{uo-closure}, it suffices
to prove $\overline{O_f}^o = \overline{O_f}^{uo}$, or equivalently,
$\overline{O_f}^{uo} \subset \overline{O_f}^{o}$, since the reverse inclusion is
clear.
Take any $g\in  \overline{O_f}^{uo}$. Without loss of generality, assume $g\geq
0$.
By Remark~\ref{uo-cl-re}\eqref{uo-cl-re1}, $g$ is the a.s.-limit of a sequence
in $O_f$.
For each $n\in \N$, let $g_n = g\wedge n\one$.  Clearly, each $g_n$ is a bounded
function in $X_+$ and is  the a.s.-limit of a sequence in $O_f$.
By Lemma~\ref{split}, we find $h_n^1, h_n^2\in X_+$ and a set $A_n\in
\Sigma$ such that $\mathbb{P}(\Omega\backslash A_n) \leq \frac{1}{2^n}$,  $\supp
h_n^1\subseteq  A_n$, $$\supp
h^2_n\subseteq \Omega\backslash A_n,\;\;
\norm{(g_n-h_n^1)\one_{A_n}}_\infty <\frac{1}{2^n},\eqno(1)$$
$\|h^2_n\|_X < \frac{1}{2^n}$
and $h_n = h_n^1 + h_n^2 \in O_f$.
Let $B_n = \{g\leq n\}\bigcap (\cap^\infty_{m=n}A_m)$.
Then by (1),
\[ \|(g-h_n)\one_{B_n}\|_\infty \leq
\|(g_n-h_n)\one_{A_n}\|_\infty=\norm{(g_n-h_n^1)\one_{A_n}}_\infty\leq
\frac{1}{2^n}.\eqno(2)\]
Since $B_n\uparrow$ and $\mathbb{P}(B_n)\rightarrow1$, it follows from (2) that
$h_n\xrightarrow{a.s.}g$.
Since $\supp
h_n^1\subseteq  A_n$, we have $0\leq h_n^1 \leq g + \one \in X$ by (1). Since $h: =
\sum_n h^2_n$ converges in
$X$, it follows that $0\leq h_n \leq g+\one+ h\in X$ for all $n$, so that $(h_n)$
is order bounded in $X$.
Therefore, $h_n \stackrel{o}{\to}g$ and  $g\in \overline{O_f}^{o}$. This proves
the special case.

For the general case, assume $0<x\in X^a$ and $y>0$. Let $B$ and $I$ be
the band and norm closed ideal generated by $x$, respectively. 
Since $I\subset X^a$, $I$ is an order continuous
Banach lattice. 
Thus we can regard $I$ as an ideal over some probability space
$(\Omega,\Sigma,\mathbb{P})$ with
$x$ corresponding to $\one$ (\cite[Theorem~1.b.14]{LZ:79}).
Clearly, $L^0(\Sigma)$ is the universal completion of $I$, and since $I$ is
order dense in $B$, we can view $B$ as an order dense sublattice of
$L^0(\Sigma)$ (\cite[Theorem~23.21]{AB:78}). Using order denseness of $B$ in
$L^0(\Sigma)$, $\sigma$-order completeness of $B$ and the countable sup property
of
$L^0(\Sigma)$, it is straightforward to verify that $B$ is order complete, and
thus is an ideal of $L^0(\Sigma)$ (\cite[Theorem~2.2]{AB:78}).
Therefore, $B$ is a Banach function space over $(\Omega,\Sigma,\mathbb{P})$,
and $\one=x\in B^a$.

Suppose now $z\in \overline{\overline{L_{x,y}}^o}^o$. Without loss of
generality, assume $z\geq 0$. Let $P$ be the band projection from $X$ onto $B$.
Since $P$ is a lattice homomorphism,
$P(L_{x,y})=L_{Px,Py}=L_{x,Py}=L_{\one, Py}=O_{Py}$. Moreover, since $P$ is
order continuous,
it follows that $$Pz\in
\overline{\overline{P(L_{x,y})}^{o'}}^{o'}=\overline{\overline{O_{Py}}^{o'}}^{o'
}=\overline{O_{Py}}^{o'},$$ where $o'$ indicates the order closure is taken in
$B$, and the last equality follows from the previous case. Note that $B$ has the
countable sup property, so that we can find a positive sequence $(w_n)$ in
$O_{Py}$ such that $w_n\xrightarrow{o} Pz$ in $B$.
We may write $w_n=Pu_n$, where $0\leq u_n\in L_{x,y}$. Clearly, $(Pu_n)$ is
order bounded, say, $Pu_n\leq a$ for all $n\geq 1$ and some $a\in X_+$. 
Then it follows from
\begin{align*}
 &\abs{P(u_n\wedge nx)-Pz}\leq \abs{Pu_n-Pz}+\abs{Pu_n-P(u_n\wedge nx)}\\
=&\abs{Pu_n-Pz}+\abs{Pu_n-(Pu_n)\wedge (nx)}=\abs{Pu_n-Pz}+(Pu_n-nx)^+\\
\leq &\abs{Pu_n-Pz}+(a-nx)^+
\end{align*}
that 
$$P(u_n\wedge nx)\xrightarrow{o}Pz\mbox{ in }X.$$
Note that $I-P$ is also a lattice homomorphism and  $(I-P)x=0$. Therefore,
$(I-P)u\in \Span(I-P)y$ for any $u\in L_{x,y}$. It follows that $(I-P)z\in
\Span(I-P)y$ as well, say, $(I-P)z=\lambda (I-P)y$. Now put $z_n=u_n\wedge
nx+\lambda (y-y\wedge nx)\in L_{x,y}$. Since $y\wedge nx\uparrow y$,
$Pz_n\xrightarrow{o}Pz$. Clearly, $(I-P)z_n=\lambda (I-P)y=(I-P)z$. Hence, $z_n\xrightarrow{o} z$
in $X$, so that $z\in \overline{L_{x,y}}^o$. This proves that $
\overline{L_{x,y}}^o$ is order closed.
\end{proof}

Orlicz spaces have been used in mathematical finance and economics as a general
framework of model spaces; see, e.g., \cite{BF:10, CL:09,GLMX:17,GLX:16,GX:16}.
We state Theorem~\ref{option} in this setting. We refer to \cite[Chapter
2]{ES:02} for definitions of Orlicz functions and spaces.

\begin{corollary}
The order closure of the option space $O_f$ is order closed for every $f\geq 0$
in an Orlicz space $L^\Phi$ over a probability space.
\end{corollary}
\begin{proof}
If $\Phi$ is finite-valued then it is well known that $\one\in(L^\Phi)^a$, so
that Theorem~\ref{option} applies. If $\Phi$ is not finite-valued, then
$L^\Phi=L^\infty$.
If a sequence $(g_n)$ in $O_f$ converges a.s.~to some $g$, then $O_f\ni
(g_n\wedge M\one)\vee(-M\one)\xrightarrow{o}g$ in $L^\infty$, where
$M=\norm{g}_\infty$.
\end{proof}

\begin{example}There exists a Banach function space $X$ for which
$\overline{O_f}^o$ is not order closed for some $f\geq 0$. Indeed, take any
Banach function space $X'$ which is not order continuous.
Then by Theorem~\ref{o-closure}, we can find $x,y>0$ such that
$\overline{L_{x,y}}^o$ is not order closed. 
Replacing $y$ with $x+y$, we may assume that $0<x<y$. By restricting to $\supp
y$, we may assume $y>0$ a.s.
Then $X := \{\frac{f}{y}: f\in X'\}$ with the norm $\|\frac{f}{y}\|_X :=
\|f\|_{X'}$ is a Banach function space,  $\one, x/y\in X$, and
$\overline{O_{x/y}}^o$ is not order closed in $X$. 
\end{example}

Our next result says that $\overline{Y}^o$ is o-closed when $Y$ is regular.
The
following lemma is well-known and was also observed in \cite{DL:05}.

\begin{lemma}\label{o-closed-sup}
Let X be an order complete vector lattice, and Y be a sublattice of X. Then Y is
order closed in $X$ if and only if for any subset $A$ of $Y_+$, its supremum in
$X$, whenever existing, belongs to $Y$.
\end{lemma}

\begin{theorem}\label{o-closure-reg}
Let $X$ be a vector lattice and $Y$ be a regular sublattice of $X$. Then 
$\overline{Y}^{uo}=\overline{Y}^{o}$, and both are order closed.
Moreover, if $0<x\in\overline{Y}^{o}$, then there exists a net $(y_\alpha)$
in
$Y_+$ such that $y_\alpha\uparrow x$ in $X$.
\end{theorem}

\begin{proof}
First assume that $X$ and $Y$ are both order complete.
For any $0\leq x\in \overline{Y}^{o}$, take $(y_\alpha)$ in $Y_+$ such that $
y_\alpha\xrightarrow{o}x$ in $X$. We may assume that
$(y_\alpha)$ is order bounded in $X_+$. Then $x=\sup_\alpha\inf_{\beta\geq
\alpha}y_\beta$, where the inf and sup are taken in $X$. Note that the infimum
of $(y_\beta)_{\beta\geq \alpha}$ exists in $Y$ by order completeness of $Y$,
and equals the infimum of $(y_\beta)_{\beta\geq \alpha}$ in $X$, by regularity
of $Y$. Put $z_\alpha=\inf_{\beta\geq \alpha}y_\beta$.
Then $$(z_\alpha)\subset  Y_+,\;\sup z_\alpha=x,\eqno(\circ)$$ 
where the supremum is taken in $X$.
Now pick any subset $A$ of $ (\overline{Y}^o)_+$ which has a supremum $x$ in
$X$. For any
$a\in A$, we can find, by $(\circ)$, a set $A_a$ in $Y_+$ such that $\sup A_a=a$
in $X$. It
is clear that $\sup\cup_{a\in A}A_a=x$ in $X$. Adding finite suprema to
$\cup_{a\in A}A_a$, we obtain a net in $Y_+$ which increases to $x$ in
$X$, whence $x\in \overline{Y}^{o}$. This proves that $\overline{Y}^o$ is
order closed in $X$ by Lemma~\ref{o-closed-sup}.

In general, by \cite[Theorem~2.10]{GTX:16},  the order
completion $Y^\delta$ of $Y$ is a regular sublattice of the order completion
$X^\delta$ of $X$. Therefore, by the preceding case, the order closure 
$\overline{Y^\delta}^o$ of $Y^\delta$ in $X^\delta$ is an order
closed sublattice in $X^\delta$. We claim that $$\overline{Y^\delta}^{o} \cap
X=\overline{Y}^o.$$
For any $x\in \overline{Y}^o$, there exists a net $(y_\alpha)$ in $Y$ such that
$y_\alpha \xrightarrow{o}x$ in $X$.  
By \cite[Corollary~2.9]{GTX:16}, we have
$y_\alpha \xrightarrow{o}x$ in $X^\delta$, and therefore, $x\in
\overline{Y^\delta}^{o}$. It follows that $\overline{Y}^{o}\subset
\overline{Y^\delta}^{o} \cap X$.
Conversely, pick $x\in  \overline{Y^\delta}^{o} \cap X$. Without loss of
generality, assume $x\geq 0$. By $(\circ)$, we can find a subset $A\subset
(Y^\delta)_+$ such that $x=\sup A$ in $X^\delta$. By order denseness of $Y$ in
$Y^\delta$, for each $a\in A$, we can find $A_a\subset Y_+$ such that $a=\sup
A_a$ in $Y^\delta$, and therefore, in $X^\delta$, by regularity of $Y^\delta$ in
$X^\delta$. It follows that $\cup_{a\in A}A_a$ is a subset of $Y_+$ and its
supremum equals $x$ in $X^\delta$, and therefore, in $X$.
Adding finite suprema into  $\cup_{a\in A}A_a$ yields a net in $Y_+$ which
increases  to $x$ in $X$; hence, $x\in
\overline{Y}^o$. 
This proves the claim.
Finally, let $(y_\alpha)$ be a net in $\overline{Y}^o$ and $x\in X$ such that
$y_\alpha\xrightarrow{o}x$ in $X$. Then $(y_\alpha)\subset
\overline{Y^\delta}^o$ by the claim, and by $y_\alpha\xrightarrow{o} x$ in
$X^\delta$, it follows from order closedness of $\overline{Y^\delta}^o$ that
$x\in \overline{Y^\delta}^o$. Therefore, $x\in \overline{Y}^o$, by the claim
again. This proves that $ \overline{Y}^o$ is order closed.
\end{proof}

\section{Measurability}\label{measurability}

In \cite{DL:05,DL:05b}, Luxemburg and de Pagter related order closed sublattices
to
measurability in vector lattices. Using their result, we can provide another
approach to obtain smallest order closed sublattices. 
We first recall some
definitions from \cite{DL:05}.
Let $X$ be an order complete vector lattice with a weak unit $u>0$. 
The set $C_u$ collects all components of $u$, i.e., all $x\in X$ such that
$(u-x)\wedge x=0$.
A subset
$\mathscr{F}$ of $C_u$ is called a \emph{complete Boolean subalgebra
of $C_u$} if $0\in \mathscr{F}$, $u-a\in\mathscr{F}$ for any $a\in \mathscr{F}$,
 and $\sup \mathscr{C}\in\mathscr{F}$ for any subset $\mathscr{C}$ of
$\mathscr{F}$. 
For such $\mathscr{F}$, a vector $x \in X$ is said to be \emph{measurable with
respect to $\mathscr{F}$}  if $P_{(\lambda u-x)^+}u=\sup_{n\ge 1}(n(\lambda
u-x)^+ )\wedge u \in \mathscr{F}$ for all $\lambda \in \mathbb{R}$. 
Denote by $L_0(\mathscr{F})$ the collection of all elements in $X$ that are measurable
with
respect to $\mathscr{F}$. For a subset $A$ of $ X$, denote by $\sigma(A)$ the
intersection of all
complete Boolean subalgebras of $C_u$ with respect to which each $a\in A$ is
measurable; clearly, it is the smallest such complete Boolean subalgebra of
$C_u$.

\begin{example}
Given a probability space $(\Omega,\Sigma,\mathbb{P})$, for a
$\sigma$-subalgebra $\mathscr{A}$ of $\Sigma$, put
$\mathscr{F}=\{\one_F:F\in\mathscr{A}\}$. Then a simple application of the
countable sup property of $L_0(\Sigma)$
yields that $\mathscr{F}$ is a complete Boolean
subalgebra of $C_\one$ in $L_0( \Sigma )$. Conversely, if $\mathscr{F}$
is a complete Boolean
subalgebra of $C_\one$ in $L_0( \Sigma )$, then
$\mathscr{A}:=\{F\in\Sigma: \one_F \in \mathscr{F}\}$ is a $\sigma$-subalgebra
of $\Sigma$.
\end{example}

\begin{theorem}[{\cite{DL:05}}]\label{dl-theorem}
Let $X$ be an order complete vector lattice with a weak unit $u>0$, and
$\mathscr{F}$ be a complete Boolean subalgebra of $C_u$.
Then ${L}_0(\mathscr{F})$ is an order closed sublattice of $X$.
\end{theorem}

The next proposition is a more precise version of Theorem~\ref{dl-theorem}.

\begin{proposition}\label{ocs}
Let $X$ be an order complete vector lattice with a weak unit $u>0$, and $A$ be a
subset of $X$. Then
${L}_0(\sigma(A))$ is the smallest order closed sublattice containing $A$ and
$u$.
\end{proposition}

\begin{proof}
It is clear that $u\in \sigma(A)\subset L_0(\sigma(A))$.
Since each $a\in A$ is measurable with respect to $\sigma(A)$,  it is also
immediate that $A\subset L_0(\sigma(A))$. 
Now let $Y$ be an order closed sublattice of $X$ containing $A$ and $u$. 
Put $\mathscr{F}=Y \cap {C}_u$. We first claim that $\mathscr{F}$ is a complete
Boolean subalgebra of $C_u$.  Indeed, it is clear that $\mathscr{F}\subset
{C}_u$, $0\in \mathscr{F}$, and if $v\in \mathscr{F}$ then $u-v\in
\mathscr{F}$. Now if $\mathscr{C}\subset \mathscr{F}$, then the supremum of
$\mathscr{C}$ in $X$ belongs to $Y$ by Lemma~\ref{o-closed-sup}.
By \cite[Theorem~1.49]{AB:06},  the supremum of $\mathscr{C}$ in $X$
also belongs to ${C}_u$, and therefore,  to $\mathscr{F}$. This
proves the claim. Next, we show that $\sigma(A) \subset
\mathscr{F}$, or equivalently, that each $a\in A$ is measurable with
respect to $\mathscr{F}$. Indeed, for any $a \in A$, any $\lambda\in\mathbb{R}$
and any $n\geq 1$, we have $n(\lambda u-a)^+ \wedge u\in Y$. Therefore, their
supremum in $X$, over all $n\in \N$, also belongs to $Y$, by
Lemma~\ref{o-closed-sup}. Note that this supremum is simply $P_{(\lambda u-a)^+}
u$; we thus obtain that $P_{(\lambda u-a)^+ }u\in Y\cap
{C}_u=\mathscr{F}$. This proves that $a$ is measurable with respect to
$\mathscr{F}$, as desired.
Finally, for any $x\in L_0(\sigma(A))$, by \cite[Proposition~2.6]{DL:05}, there
exists a sequence $(x_n)$ in $\Span(\sigma(A))$ such that $x_n\xrightarrow{o} x$ in $X$. Since $\sigma(A) \subset \mathscr{F} \subset Y $, we have $x_n \in
Y$ for each $n$. It follows from order closedness of $Y$ that $x\in Y$. Hence,
$L_0(\sigma(A))\subset Y$.
\end{proof}

A combination of Proposition~\ref{ocs} and Theorem~\ref{uo-closure} (cf.~ also Remark~\ref{uo-cl-re})  immediately
gives $(*)$ in the Introduction. 
In fact, using Proposition~\ref{ocs} and
Theorem~\ref{option}, we obtain the following strong spanning power of options.

\begin{corollary}\label{option-strong}
Let $X$ be a Banach function space over a probability space such that $\one\in
X^a$, and let $f\in X_+$.
Then $X(\sigma(f))=\overline{O_f}^o$.
\end{corollary}

The following is immediate by Proposition~\ref{ocs} and
Theorem~\ref{o-closure-reg}.

\begin{corollary}\label{combination}
Let X be an order complete vector lattice with a weak unit $u>0$, and $Y$ be a
regular sublattice of $X$ containing $u$. Then $x\in L_0(\sigma(Y))$ iff there
exists a net $(y_\alpha)$ in $Y$ such that $y_\alpha\xrightarrow{o}x$ in
$X$. If, in addition, $x>0$, then $(y_\alpha)$ can be chosen positive and
increasing.
\end{corollary}

For a complete Boolean subalgebra $\mathscr{F}$ of $C_u$, it is easily seen that
$\Span(\mathscr{F}) $ is a regular sublattice in $X$ and
$\sigma(\Span(\mathscr{F}))=\mathscr{F}$. Thus, Corollary~\ref{combination} can
be viewed as a generalization of \cite[Proposition~{2.6}]{DL:05} which is
essentially Freudenthal's Spectral Theorem.

The following is also immediate by Proposition~\ref{ocs} and extends
\cite[Lemma~2.2]{GX:17}.

\begin{corollary}\label{closed-mea}
Let X be an order complete vector lattice with a weak unit $u>0$, and $Y$ be a
sublattice of $X$ containing $u$. Then $Y$ is order closed if and only if
$Y=L_0(\sigma(Y))$.
\end{corollary}

\begin{example}
Let $Y$ be an order closed sublattice in a Banach function space $X$ over 
$(\Omega,\Sigma,\mathbb{P})$.
Then there exist $u\in Y_+$ and a $\sigma$-subalgebra
$\mathcal{G} $ of $\Sigma$ such that $$Y=\{g\in X:g=uh,\;
h\text{ is }\mathcal{G}\text{-}measurable\}.$$
Indeed, it is known that $X$ has a weak unit $e$. 
By the countable sup property of $X$, one can take a sequence $(g_n)$ in $Y_+$
such
that $\sup_n(g_n\wedge e)=\sup_{g\in Y_+}(g\wedge e)$ in $X$. 
Then $\sum_1^N\frac{g_n}{2^n\norm{g_n}+1}\uparrow u$ for some $u\in X$. Clearly,
$u\in Y_+$, and $\mathbb{P}(\supp g\backslash \supp u)=0$ for any $g\in
Y$. 
Thus by passing to the support of $u$, one may assume that $u$ is a weak unit of
$X$. 
By Corollary~\ref{closed-mea}, we have $Y=L_0(\sigma(Y))$, where $\sigma(Y)$ is
the complete Boolean subalgebra generated by $Y$ in $C_u$. 
Every member in $C_u$ has the form $\one_Au$ for some set $A\in \Sigma$. 
Collecting all such $A$ together for the members in $\sigma(Y)$ forms a
$\sigma$-subalgebra of $\Sigma$, which we denote by $\mathcal{G}$.
Now for each $0\leq g\in L_0(\sigma(Y))$, by \cite[Proposition~{2.6}]{DL:05}, 
there exists a sequence $(g_n) $ in $\Span\sigma(Y)$ such that $0\leq
g_n\uparrow g$ in $X$. 
Of course, $g_n=h_nu$ where $h_n$ is a simple function on $\mathcal{G}$, and
$0\leq h_n\uparrow$. Let $h=\lim_nh_n$. 
Then $h$ is measurable with respect to $\mathcal{G}$, and $g=uh$. The reverse
inclusion can be proved similarly.
\end{example}

\end{document}